\documentclass[12pt]{amsart} 

\usepackage{amsmath,amsthm,amstext,amssymb,amsfonts,latexsym} 
\usepackage{graphicx} 
\usepackage{enumerate}
\usepackage{xcolor} 
\usepackage[hidelinks]{hyperref} 

\allowdisplaybreaks[1] 

\newtheorem{thm}{Theorem}[section]
\newtheorem{lem}[thm]{Lemma}
\newtheorem{prop}[thm]{Proposition}

\theoremstyle{definition} 
\theoremstyle{remark} \newtheorem{rem}[thm]{Remark}

\newcommand{\tr}{\mathrm{tr}\,}

\newcommand{\bbr}{\mathbb{R}}
\newcommand{\bbs}{\mathbb{S}}
\newcommand{\abs}[1]{\lvert #1\rvert}

\begin{document}

\title[Uniqueness of self-similar solutions]
{Uniqueness of self-similar solutions to flows by quotient
curvatures}

\author{Li Chen}
\address{Faculty of Mathematics and Statistics, Hubei University, Wuhan
430062,  P.R. China. } \email{chenli@hubu.edu.cn}

\author{Shanze Gao}
\address{School of Mathematical Science, University of
Science and Technology of China, Hefei Anhui,
        230026 , P.R. China}
\email{shanze@ustc.edu.cn}

\thanks{
This research was supported in part by Hubei Key Laboratory of
Applied Mathematics (Hubei University).}

\date{}
\maketitle
\begin{abstract}
In this paper, we consider a family of closed hypersurfaces which shrink self-similarly with speed of quotient curvatures. We show
that the only such hypersurfaces are shrinking spheres.
\end{abstract}

\maketitle {\it \small{{\bf Keywords}: Uniqueness, convex solutions,
quotient of curvatures, self-similar.}

{{\bf MSC}: 35J15, 35J60, 53C44.} }

\section{Introduction}

Let $X: \mathcal{M} \rightarrow \mathbb{R}^{n+1}$ be a smooth closed hypersurface with $n\geq 2$, satisfying
\begin{equation}\label{eq:SS}
\mathcal{F}(\kappa)=\langle X, \nu\rangle,
\end{equation}
where $\kappa=(\kappa_{1},\kappa_{2},...,\kappa_{n})$ denotes the principal
curvatures of $\mathcal{M}$, $\mathcal{F}$ is a homogeneous symmetric
function of $\kappa$ and $\nu$ denotes the outward normal vector of $\mathcal{M}$. Such hypersurfaces are called the self-similar solutions to the following curvature flow
\begin{equation}\label{eq:CF}
\frac{\partial }{\partial t}X=-\mathcal{F}\nu
\end{equation}
(see \cite{Huisken1990,McCoy2011,Gao-Li-Ma2018} etc.).

Self-similar solutions play an important role in describing
asymptotic behaviors of curvature flows such as mean curvature flow
and Gauss curvature flow (see \cite{Huisken1990,Guan-Ni2017,Andrews-Guan-Ni2016}
etc.). Examples in \cite{Abresch-Langer1986,Angenent1989} show that
the solution is usually not unique. In 1990,
Huisken\cite{Huisken1990} proved that the closed self-similar
solution to mean curvature flow is a sphere under mean convexity
condition. In \cite{Andrews1994,Andrews2007}, Andrews studied
curvature flows \eqref{eq:CF} motioned by a class of
$1$-homogeneous functions of curvatures, including $\mathcal{F}=(\sigma_{k}/\sigma_{l})^{\frac{1}{k-l}}$ where $\sigma_{k}$ is the $k$-th elementary symmetric polynomial and $0\leq l<k\leq n$. Later,
McCoy\cite{McCoy2011} showed the uniqueness of self-similar
solutions to these flows by elliptic methods.

Recently, the uniqueness of strictly
convex self-similar solutions to $\alpha$-Gauss curvature flow is
proven by Choi-Daskalopoulos\cite{Choi-Daskalopoulos} and
Brendle-Choi-Daskalopoulos\cite{Brendle-Choi-Daskalopoulos}. In details, they showed if $\mathcal{M}$ is a strictly convex
hypersurface in $\mathbb{R}^{n+1}$ which satisfies the equation
\begin{equation*}
K^{\alpha}=\langle X, \nu\rangle,
\end{equation*}
then $\mathcal{M}$ is a sphere when $\alpha>\frac{1}{n+2}$, where
$K$ is the Gauss curvature. In \cite{Choi-Daskalopoulos,
Brendle-Choi-Daskalopoulos}, they introduced two important
functions which can be written as
\begin{equation}\label{eq:W}
W(x)=u \cdot \lambda_{\max}(b_{ij})
-\frac{n\alpha-1}{2n\alpha}(u^2+|Du|^2)
\end{equation}
and
\begin{equation}\label{eq:Old-Z}
\tilde{Z}(x)=u\cdot \tr(b_{ij})-\frac{n\alpha-1}{2\alpha}(u^2+|Du|^2),
\end{equation}
where $u$ is the support function of
$\mathcal{M}$, $b_{ij}=u_{ij}+u\delta_{ij}$ and $\lambda_{\max}(b_{ij})$ is the largest eigenvalue of $(b_{ij})$ (see details in Section \ref{sec:pre}).

Later, motivated by the idea of Choi-Daskaspoulos
\cite{Choi-Daskalopoulos} and Brendle-Choi-Daskaspoulos
\cite{Brendle-Choi-Daskalopoulos}, Li, Ma and the second author
\cite{Gao-Li-Ma2018} proved the uniqueness of strictly convex
self-similar solutions to a class of curvature flows \eqref{eq:CF},
which includes $\mathcal{F}=\sigma_{k}^{\alpha}$ for $\alpha>\frac{1}{k}$. But the following case
\begin{equation}\label{eq:curq}
\left(\frac{\sigma_{k}(\kappa)}{\sigma_{l}(\kappa)}\right)^{\alpha}=\langle
X,\nu \rangle
\end{equation}
is not included in their paper, where $1\leq l<k\leq n$ and $\alpha>\frac{1}{k-l}$. We remark that the uniqueness of \eqref{eq:curq} is proven in \cite{Andrews1994,Andrews2007,McCoy2011} when $\alpha=\frac{1}{k-l}$.

To overcome this difficulty, the first author \cite{ChenLi2019}
introduced
 a new $Z$ function which is defined as follows
\begin{equation}\label{eq:New-Z}
Z=uG(b_{ij})-\frac{n\beta}{2} (u^{2}+\abs{Du}^{2}),
\end{equation}
where $\beta$ is a constant to be chosen later, and
$$G=\frac{n}{k}\bigg(\sigma_{1}(b_{ij})-(k+1)\frac{\sigma_{k+1}(b_{ij})}
{\sigma_{k}(b_{ij})}\bigg).$$ Using this new $Z$ function
\eqref{eq:New-Z} together with the $W$ function \eqref{eq:W}, he proved
any closed strictly convex hypersurface in $\mathbb{R}^{n+1}$, satisfying
the equation
\begin{equation*}\label{Chen}
\left(\frac{\sigma_{n}(\kappa)}{\sigma_{n-k}(\kappa)}\right)^{\alpha}=\langle
X,\nu \rangle,
\end{equation*}
is a sphere when $\alpha>\frac{1}{k}$.

In this paper, using the new $Z$ function \eqref{eq:New-Z} and the $W$
function \eqref{eq:W}, we thoroughly prove uniqueness of solutions to the equation \eqref{eq:curq}.
\begin{thm}\label{th:curq}
Let $\mathcal{M}$ be a closed strictly
 convex hypersurface in $\bbr^{n+1}$, which satisfies
\begin{equation}\label{eq:main}
\left(\frac{\sigma_{k}(\kappa)}{\sigma_{l}(\kappa)}\right)^{\alpha}=\langle X,\nu \rangle
\end{equation}
where $0\leq l<k\leq n$, then $\mathcal{M}$ is a standard sphere
for $\alpha>\frac{1}{k-l}$.
\end{thm}

This paper is organized as follows. In Section \ref{sec:pre},
we give some notations, recall some basic properties of convex hypersurfaces and derive basic formulas. In section \ref{sec:W},
we consider $W$ at its maximum points for a general equation. In the last
section, we prove the main theorem.

\section{Preliminaries}
\label{sec:pre}

We first recall some basic properties of convex hypersurfaces.

Let $\mathcal{M}$ be a smooth, closed, uniformly convex hypersurface
in $\mathbb{R}^{n+1}$.
The support function $u: \mathbb{S}^n\rightarrow \mathbb{R}$ of
$\mathcal{M}$ is defined by
\begin{equation*}
u(x)=\sup\{\langle x, y\rangle: y \in \mathcal{M}\}.
\end{equation*}
In this case, the supremum is attained at a point $y$ if $x$ is the outer
normal of $\mathcal{M}$ at $y$. It is well-known that (see \cite{Andrews1994} for example)
\begin{equation*}
y=u(x)x+Du(x).
\end{equation*}
Hence
\begin{equation*}\label{eq:|y|}
|y|=\sqrt{u^2+|Du|^2}.
\end{equation*}
And the principal radii of curvature of $\mathcal{M}$, under a smooth local
orthonormal frame on $\mathbb{S}^n$, are the eigenvalues of the
matrix $(b_{ij})$ where $b_{ij}=u_{ij}+u\delta_{ij}$. Thus, we can rewrite the two important functions $W$ and $\tilde{Z}$ in \cite{Choi-Daskalopoulos,Brendle-Choi-Daskalopoulos} as \eqref{eq:W} and \eqref{eq:Old-Z}. From the relation between principal curvatures and principal radii, we know
\begin{equation*}
\sigma_{k}(\kappa)=\frac{\sigma_{n-k}(b_{ij})}{\sigma_{n}(b_{ij})}
\end{equation*}
and
we can rewrite the equation \eqref{eq:main} by the support function $u$
of $\mathcal{M}$. So Theorem \ref{th:curq} is equivalent  to the
following theorem.
\begin{thm}\label{th:curq-}
Any smooth strictly convex solution of the following equation
\begin{equation}\label{eq:cmq}
\frac{\sigma_{k}(b_{ij})}{\sigma_{l}(b_{ij})}=u^{p-1}\quad \text{ on }\bbs^{n}
\end{equation}
is a constant for $1>p>1-{k-l}$, where the strict convexity of a
solution $u$ means that the matrix $(b_{ij})$ positive definite on $\mathbb{S}^n$ and $0\leq l<k\leq n$.
\end{thm}

\begin{rem}
When $l=0$, equation \eqref{eq:cmq} is $k$-th $L_{p}$-
Christoffel-Minkowski problem with constant right hand side. In this
case, Theorem \ref{th:curq-} is proved by the first author in
\cite{ChenLi2019}.
\end{rem}

Throughout this paper, we do calculations in a unit orthogonal frame
and use summation convention unless otherwise stated. Let $D$ denotes
the covariant derivative with respect to the standard metric of the sphere $\mathbb{S}^n$ and $R_{ijkl}$ denote the Riemannian curvature tensor of $\mathbb{S}^n$. And
$u_{ij}=D_{i}D_{j}u$, $u_{ijk}=D_{k}u_{ij}$ and
$u_{ijkl}=D_{l}u_{ijk}$.
From Ricci identity, we know
\begin{align*}
u_{ijk}&=u_{ikj}+u_{m}R_{mijk}=u_{ikj}+u_{m}(\delta_{mj}\delta_{ik}-\delta_{mk}\delta_{ij})\\
&=u_{ikj}+u_{j}\delta_{ik}-u_{k}\delta_{ij}.
\end{align*}
This implies $b_{ijk}=b_{ikj}$. Furthermore,
\begin{align*}
b_{ijkl}&=b_{ijlk}+b_{mj}R_{mikl}+b_{im}R_{mjkl}\\
&=b_{ijlk}+b_{kj}\delta_{il}-b_{lj}\delta_{ik}+b_{ik}\delta_{jl}-b_{il}\delta_{jk}.
\end{align*}
This implies
\begin{equation}\label{eq:biijj}
b_{iijj}=b_{ijij}=b_{ijji}-b_{jj}+b_{ii}=b_{jjii}-b_{jj}+b_{ii}
\end{equation}
for any fixed $i,j$. Since

For the convenience of discussion, instead of \eqref{eq:cmq}, we consider
\begin{equation}\label{eq:F}
F(u_{ij}+u\delta_{ij})=u^{p_{0}}\quad \text{on }\bbs^{n}
\end{equation}
where $F$ is an $1$-homogeneous function, i.e. $F(tA)=tF(A)$. For
any $1$-homogeneous function $G=G(b_{ij})$, it is easy to check the
following equation by \eqref{eq:biijj},
\begin{equation}\label{eq:G}
\begin{aligned}
F^{ij}D_{i}D_{j}G&=G\sum_{i}F^{ii}- F\sum_{i}G^{ii}+G^{ij}D_{i}D_{j}F\\
&\quad-G^{kl}F^{ij,pq}b_{ijk}b_{pql}+F^{ij}G^{kl,pq}b_{kli}b_{pqj},
\end{aligned}
\end{equation}
where $F^{ij}=\frac{\partial F}{\partial b_{ij}}$ and $F^{ij,pq}=\frac{\partial^{2} F}{\partial b_{ij}\partial b_{pq}}$.

By direct calculations, we have
\begin{align}
&F^{ij}D_{i}D_{j}u=F-u\sum_{i}F^{ii}.\label{eq:u}\\ &F^{ij}D_{i}D_{j}(u^{2}+\abs{Du}^{2})=2F^{ij}b_{ik}b_{kj}-2uF+2u_{i}D_{i}F.\label{eq:u2Du2}
\end{align}

\begin{prop}\label{th:Z}
Suppose that $u$ is a solution to \eqref{eq:F} and $G=G(b_{ij})$ is
any $1$-homogeneous function. Then, for
$$Z=uG-\frac{n\beta}{2} (u^{2}+\abs{Du}^{2}),$$ we have the following
formula:
\begin{align*}
F^{ij}D_{i}D_{j}Z&=(1+p_{0}) FG-n\beta F^{ij}b_{ik}b_{kj}
+(n\beta-(1+p_{0})\sum_{i}G^{ii}) uF\\
&\quad +p_{0}(p_{0}-1)u^{p_{0}-1}G^{ij}u_{i}u_{j}-n\beta p_{0}
u^{p_{0}-1}\abs{Du}^{2}+2F^{ij}u_{i}D_{j}G\\
&\quad-uG^{kl}F^{ij,pq}b_{ijk}b_{pql}+uF^{ij}G^{kl,pq}b_{kli}b_{pqj}.
\end{align*}
\end{prop}

\begin{proof}
From \eqref{eq:G}, \eqref{eq:u} and \eqref{eq:u2Du2}, we have
\begin{align*}
F^{ij}D_{i}D_{j}Z&=(F-u\sum_{i}F^{ii})G+uG\sum_{i}F^{ii}-uF\sum_{i}G^{ii}+uG^{ij}D_{i}D_{j}F\\
&\quad-uG^{kl}F^{ij,pq}b_{ijk}b_{pql}+uF^{ij}G^{kl,pq}b_{kli}b_{pqj}\\
&\quad -n\beta(F^{ij}b_{ik}b_{kj}-uF+u_{i}D_{i}F)+2F^{ij}u_{i}D_{j}G.
\end{align*}

Using \eqref{eq:F}, we know
\begin{align*}
D_{i}F&=p_{0}u^{p_{0-1}}u_{i},\\
D_{i}D_{j}F&=p_{0}u_{0}^{p_{0}-1}(b_{ij}-u\delta_{ij})+p_{0}(p_{0-1})u^{p_{0}-2}u_{i}u_{j}.
\end{align*}

We finish the proof by combining above equations together.

\end{proof}

\section{Analysis at a maximum point of $W$}
\label{sec:W}

To study
$$W=u\lambda_{\max}(b_{ij})-\frac{\beta}{2}(u^{2}+\abs{Du}^{2}),$$
we need the following lemma which is a
slight modification of Lemma 5 in \cite{Brendle-Choi-Daskalopoulos}.
\begin{lem}\label{th:phi}
We choose a unit orthogonal frame such that $(b_{ij})=\mathrm{diag}(b_{11},b_{22},...,b_{nn})$ at a point $\bar{x}\in \bbs^{n}$. Let $\mu$ denote the multiplicity of $b_{11}$ at $\bar{x}$, i.e., $b_{11}(\bar{x})=\cdots=b_{\mu\mu}(\bar{x})>b_{\mu+1,\mu+1}(\bar{x})\geq \cdots \geq b_{nn}(\bar{x})$. Assume that $\varphi$ is a smooth function such that $\varphi\geq \lambda_{\max}$ and $\varphi(\bar{x})=\lambda_{\max}(\bar{x})=b_{11}(\bar{x})$. Then, at $\bar{x}$, we have
\begin{enumerate}[i)]
\item $b_{kli}=D_{i}\varphi\delta_{kl}$ for $1\leq k,l\leq \mu$;
\item $D_{i}D_{i}\varphi\geq b_{11ii}+2\sum_{l>\mu}(b_{11}-b_{ll})^{-1}b_{1li}^{2}$.
\end{enumerate}
\end{lem}

\begin{proof}
See the proof of Lemma 5 in \cite{Brendle-Choi-Daskalopoulos}.
\end{proof}

Now we use maximum principle at a maximum point of $W$ as in \cite{Brendle-Choi-Daskalopoulos,Gao-Li-Ma2018,ChenLi2019}. The concavity of $\left(\frac{\sigma_{k}}{\sigma_{l}}\right)^{\frac{1}{k-l}}$ is important in this step. We write down the details in the following form.

\begin{lem}\label{th:W}
Assume that $u$ is a positive solution to
\begin{equation*}
F(b_{ij})=u^{p_{0}}\quad \text{on }\bbs^{n}
\end{equation*}
such that the matrix $(b_{ij})$ is positive definite
on $\bbs^{n}$, where $F$ is a $1$-homogeneous and concave function
respect to $(b_{ij})$ and  the matrix $(\frac{\partial F}{\partial
b_{ij}})$ is positive definite on $\bbs^{n}$. For $-1<p_{0}<0$, if
$\bar{x}\in\bbs^{n}$ is a maximum point of $W$, then
$(b_{ij})(\bar{x})$ is a scalar matrix and $Du(\bar{x})=0$.
\end{lem}

\begin{proof}

We define $\varphi$ by
$$u\varphi-\frac{p_{0}+1}{2}(u^{2}+\abs{Du}^{2})=W_{\max},$$ where
$W_{\max}$ is the maximum of $W$ on $\bbs^{n}$. This implies that
$\varphi$ satisfies the assumption in Lemma \ref{th:phi}. Using
Lemma \ref{th:phi}, we have
\begin{align*}
0&\geq (1+p_{0})\sum_{i}F^{ii}b_{ii}(b_{11}-b_{ii})+p_{0}(p_{0}-1)u^{p_{0}-1}u_{1}^{2}-uF^{ij,pq}b_{ij1}b_{pq1}\\
&\quad - (1+p_{0})p_{0}u^{p_{0}-1}\abs{Du}^{2}+2F^{ij}u_{i}b_{11j}+2uF^{ii}\sum_{l>\mu}(b_{11}-b_{ll})^{-1}b_{1li}^{2}.
\end{align*}

From $0=D_{i}W=\left(b_{11}-(1+p_{0})b_{ii}\right)u_{i}+ub_{11i}$, we know
\begin{align*}
F^{ij}u_{i}b_{11j}&=-u^{-1}\sum_{i>\mu}F^{ii}\left(b_{11}
-(1+p_{0})b_{ii}\right)u_{i}^{2}+p_{0}u^{-1}F^{11}b_{11}u_{1}^{2}
\end{align*}
in view of $b_{11i}=0$ for $1<i\leq\mu$ by Lemma \ref{th:phi}.

Using $$F^{ij,pq}b_{ij1}b_{pq1}=F^{ii,jj}b_{ii1}b_{jj1}+2\sum_{i> j}\frac{F^{ii}-F^{jj}}{b_{ii}-b_{jj}}b_{ij1}^{2}$$ and $$b_{kli}=0 \text{ for } 1\leq k,l\leq\mu \text{ and } k\neq l,$$
we have
\begin{align*}
&-F^{ij,pq}b_{ij1}b_{pq1}+2F^{ii}\sum_{l>\mu}(b_{11}-b_{ll})^{-1}b_{1li}^{2}\\
&\qquad =-F^{ii,jj}b_{ii1}b_{jj1}-2\sum_{i> \mu}\frac{F^{ii}-F^{11}}{b_{ii}-b_{11}}b_{i11}^{2}-2\sum_{i> j>\mu}\frac{F^{ii}-F^{jj}}{b_{ii}-b_{jj}}b_{ij1}^{2}\\
&\hspace{3em}+2F^{11}\sum_{l>\mu}(b_{11}-b_{ll})^{-1}b_{11l}^{2}+2F^{ll}\sum_{l>\mu}(b_{11}-b_{ll})^{-1}b_{1ll}^{2}\\
&\hspace{3em}+2F^{ii}\sum_{i>l>\mu}(b_{11}-b_{ll})^{-1}b_{1li}^{2}+2F^{ii}\sum_{l>i>\mu}(b_{11}-b_{ll})^{-1}b_{1li}^{2}\\
&\qquad =-F^{ii,jj}b_{ii1}b_{jj1}+2\sum_{i> \mu}F^{ii}(b_{11}-b_{ii})^{-1}b_{i11}^{2}+2\sum_{l>\mu}F^{ll}(b_{11}-b_{ll})^{-1}b_{1ll}^{2}\\
&\hspace{3em}+2\sum_{i> j>\mu}\frac{F^{ii}(b_{11}-b_{ii})^{2}-F^{jj}(b_{11}-b_{jj})^{2}}{(b_{jj}-b_{ii})(b_{11}-b_{ii})(b_{11}-b_{jj})}b_{ij1}^{2}.
\end{align*}

Since $F$ is concave, we know $-F^{ii,jj}b_{ii1}b_{jj1}\geq 0$.
Furthermore, combining it with
\begin{equation}\label{eq:umblic}
\sum_{i}F^{ii}b_{ii}(b_{11}-b_{ii})\geq 0
\end{equation}
and
\begin{align*}
\sum_{i> j>\mu}\frac{F^{ii}(b_{11}-b_{ii})^{2}-F^{jj}(b_{11}-b_{jj})^{2}}{(b_{jj}-b_{ii})(b_{11}-b_{ii})(b_{11}-b_{jj})}b_{ij1}^{2}\geq 0,
\end{align*}
we have
\begin{align*}
0&\geq p_{0}(p_{0}-1)u^{p_{0}-1}u_{1}^{2}-(1+p_{0})p_{0}u^{p_{0}-1}\abs{Du}^{2}\\
&\quad -2u^{-1}\sum_{i>\mu}F^{ii}\left(b_{11}-(1+p_{0})b_{ii}\right)u_{i}^{2}+2p_{0}u^{-1}F^{11}b_{11}u_{1}^{2}\\
&\quad +2u^{-1}\sum_{i> \mu}F^{ii}(b_{11}-b_{ii})^{-1}\left(b_{11}-(1+p_{0})b_{ii}\right)^{2}u_{i}^{2}\\
&=p_{0}\left(-2+2\frac{F^{11}b_{11}}{ F}\right)u^{p_{0}-1}u_{1}^{2}-(1+p_{0})p_{0}u^{p_{0}-1}\sum_{i> \mu}u_{i}^{2}\\
&\quad -\frac{2p_{0}}{u}\sum_{i> \mu}F^{ii}(b_{11}-b_{ii})^{-1}\left(b_{11}-(1+1^{-1}p_{0})b_{ii}\right)b_{ii}u_{i}^{2}.
\end{align*}

Since $-1<p_{0}<0$ and $F^{11}b_{11}<F$, the right hand-side of above inequality is non-negative which implies $Du=0$. And $b_{11}=b_{22}=\cdots=b_{nn}$ is from the equality of \eqref{eq:umblic}.
\end{proof}

\section{Proof of main theorem}
\label{sec:mainpf}

In this section, we choose
\begin{equation*}
G=\frac{n}{k}(\sigma_{1}-(k+1)\frac{\sigma_{k+1}}{\sigma_{k}})
\end{equation*}
for $F=\left(\frac{\sigma_{k}}{\sigma_{l}}\right)^{\frac{1}{k-l}}$.

It is easy to check that $n\lambda_{\max}\geq G$ which means $nW\geq Z$ and the equality occurs if and only if $(b_{ij})$ is a scalar matrix. And $G$ is convex since $\frac{\sigma_{k+1}}{\sigma_{k}}$ is concave. To estimate the right hand
side of the formula in Proposition \ref{th:Z}, we need the following lemma.

\begin{lem}\label{th:Gineq}
For $F=\left(\frac{\sigma_{k} }{\sigma_{l}}\right)^{\frac{1}{k-l}}$,
we choose
$G=\frac{n}{k}(\sigma_{1}-(k+1)\frac{\sigma_{k+1}}{\sigma_{k}})$ and
$\beta=1+p_{0}$. If $(b_{ij})$ is positive definite on $\bbs^{n}$,
then the following two inequalities hold:
\begin{enumerate}[i)]
\item $(1+p_{0}) FG-n\beta F^{ij}b_{ik}b_{kj}\geq 0$.
\item $-(1+p_{0})\sum_{i}G^{ii}+n\beta \geq 0$.
\end{enumerate}
\end{lem}

\begin{proof}
\begin{enumerate}[i)]
\item It is equivalent to show
\begin{equation}\label{eq:Ggeq}
G\geq \frac{n}{k-l}((l+1)\frac{\sigma_{l+1}}{\sigma_{l}}-(k+1)\frac{\sigma_{k+1}}{\sigma_{k}}).
\end{equation}
From Lemma 2.1 in \cite{Gao-Ma2016}
\begin{equation*}
\frac{1}{k(k-1)}\sigma_{1}-\frac{k\sigma_{k}}{(k-1)\sigma_{k-1}}+\frac{(k+1)\sigma_{k+1}}{k\sigma_{k}}\geq 0,
\end{equation*}
we know
\begin{align*}
&\frac{(l+1)\sigma_{l+1}}{l\sigma_{l}}-\frac{(k+1)\sigma_{k+1}}{k\sigma_{k}}=\sum_{i=l+1}^{k}\left(\frac{i\sigma_{i}}{(i-1)\sigma_{i-1}}-\frac{(i+1)\sigma_{i+1}}{i\sigma_{i}}\right)\\
&\qquad \leq
\sum_{i=l+1}^{k}\frac{1}{i(i-1)}\sigma_{1}=(\frac{1}{l}-\frac{1}{k})\sigma_{1},
\end{align*}
which implies \eqref{eq:Ggeq}.

\item We just need to check $\sum_{i}G^{ii}\leq n$. We show that
\begin{align*}
\sum_{i}G^{ii}=\frac{n}{k}\left(n-(k+1)(n-k)+(k+1)(n-k+1)\frac{\sigma_{k-1}
\sigma_{k+1}}{\sigma_{k}^{2}}\right) \leq n,
\end{align*}
where the inequality is from Newton's inequality.
\end{enumerate}
\end{proof}

Now, we prove Theorem \ref{th:curq-}.

\begin{proof}[Proof of Theorem \ref{th:curq-}]
First, we transform the equation \eqref{eq:cmq} to
\begin{equation*}
\left(\frac{\sigma_{k}(b_{ij})}{\sigma_{l}
(b_{ij})}\right)^{\frac{1}{k-l}}=u^{p_{0}}\quad \text{
on }\bbs^{n},
\end{equation*}
where $p_{0}=\frac{p-1}{k-l}$ and $-1<p_{0}<0$. Thus the left hand
side of equation above is a $1$-homogeneous and concave function.

The convexity of $G$ implies $F^{ij}G^{kl,pq}b_{kli}b_{pqj}\geq 0$. And, from Proposition \ref{th:Z} and Lemma \ref{th:Gineq}, we know
\begin{align*}
F^{ij}D_{i}D_{j}Z&\geq p_{0}(p_{0}-1)u^{p_{0}-1}G^{ij}u_{i}u_{j}-n\beta p_{0}u^{p_{0}-1}\abs{Du}^{2}+2F^{ij}u_{i}D_{j}G\\
&\quad-uG^{kl}F^{ij,pq}b_{ijk}b_{pql}.
\end{align*}

Using
\begin{equation}
D_{j}Z=u_{j}G+uD_{j}G-n(p_{0}+1)b_{jj}u_{j},
\end{equation}
we have
\begin{align*}
&F^{ij}D_{i}D_{j}Z-\frac{2}{u}F^{ij}u_{i}D_{j}Z\\
&\qquad\geq p_{0}(p_{0}-1)u^{p_{0}-1}G^{ij}u_{i}u_{j}-n\beta p_{0}u^{p_{0}-1}\abs{Du}^{2}\\
&\hspace{3em} -\frac{2}{u}F^{ij}u_{i}(u_{j}G-n(p_{0}+1)b_{jj}u_{j})-uG^{kl}F^{ij,pq}b_{ijk}b_{pql}\\
&\qquad =\Big\{p_{0}(p_{0}-1)FG^{ii}-n\beta p_{0}F-2GF^{ii}+2n(p_{0}+1)F^{ii}b_{ii} \Big\}\frac{u_{i}^{2}}{u}\\
&\hspace{3em}-uG^{kl}F^{ij,pq}b_{ijk}b_{pql}.
\end{align*}

If $\bar{x}$ is a maximum point of $W$, then $b_{11}=b_{22}=\cdots=b_{nn}$ by Lemma \ref{th:W}. Thus $G_{ii}(\bar{x})=1$ and $F^{ii}(\bar{x})=\frac{F}{nb_{ii}}$. This implies
\begin{align*}
&p_{0}(p_{0}-1)FG^{ii}-n\beta p_{0}F-2GF^{ii}+2n(p_{0}+1)F^{ii}b_{ii}\\
&\qquad =p_{0}(p_{0}-1)F-n(p_{0}+1) p_{0}F-2F+2(p_{0}+1)F\\
&\qquad =-(n-1)p_{0}(p_{0}+1)F>0
\end{align*}
at $\bar{x}$.

Combining with concavity of $F$, this implies that there is a small neighborhood of $\bar{x}$, denoted by $U$, such that
\begin{align*}
F^{ij}D_{i}D_{j}Z-\frac{2}{u}F^{ij}u_{i}D_{j}Z\geq 0.
\end{align*}

By $Z(\bar{x})=nW_{\max}\geq nW\geq Z$ and strong maximum principle,
we know that $W$ is a constant in $U$. Since $\bbs^{n}$ is
connected, we know that $W$ is a constant. Then Lemma \ref{th:W}
shows $Du=0$ on $\bbs^{n}$ which implies $u$ is a constant. Thus, we
complete our proof.
\end{proof}



\begin{thebibliography}{99}

\bibitem{An} Andrews, B.: Motion of hypersurfaces by Gauss curvature.
Pacific J. Math. 195 (2000),1-34.

\bibitem{Abresch-Langer1986} U. Abresch, J. Langer,
The normalized curve shortening flow and homothetic solutions. J.
Differential Geom. 23, no. 2, 175-196 (1986).

\bibitem{Andrews1994} B. Andrews, Contraction of convex hypersurfaces in Euclidean space. Calc. Var.
Partial Differ. Equ. 2, 151–171 (1994)

\bibitem{Andrews1999} B. Andrews, Gauss curvature flow: the fate of the rolling stones. Invent. Math.
138, 151–161 (1999)

\bibitem{Andrews2007} B. Andrews, Pinching estimates and motion of hypersurfaces by curvature functions. J.
Reine Angew. Math., 608, 17-33 (2007)

\bibitem{Andrews-Guan-Ni2016} B. Andrews, P.-F. Guan, L. Ni, Flow by powers of the Gauss curvature. Adv.
Math. 299, 174–201 (2016)

\bibitem{Angenent1989} S. B. Angenent, Shrinking doughnuts. Nonlinear diffusion equations and their equilibrium states, 3 (Gregynog, 1989), 21–38, Progr. Nonlinear Differential Equations Appl., 7, Birkhäuser Boston, Boston, MA, 1992.

\bibitem{Brendle-Choi-Daskalopoulos} S. Brendle, K. Choi, P. Daskalopoulos,
Asymptotic behavior of flows by powers
of the Gaussian curvature. Acta Math. 219(1), 1–16 (2017)

\bibitem{ChenLi2019} L. Chen, Uniqueness of solutions to Lp-Christoffel-Minkowski problem. arXiv:1905.11043 (2019)

\bibitem{Choi-Daskalopoulos} K. Choi, P. Daskalopoulos, Uniqueness of closed self-similar solutions to the
Gauss curvature flow. arXiv:1609.05487 (2016)

\bibitem{Gao-Ma2016} S. Z. Gao, H. Ma, Self-similar solutions of $\sigma_{k}^{\alpha}$-curvature flow.  arXiv:1611.0758 (2016)

\bibitem{Gao-Li-Ma2018} S. Z. Gao, H. Li, H. Ma, Uniqueness of closed self-similar solutions to $\sigma_{k}^{\alpha}$-curvature flow. NoDEA Nonlinear Differential Equations Appl. 25 (2018), no. 5, Art. 45, 26 pp.

\bibitem{Guan-Ni2017}P. Guan, L. Ni, Entropy and a convergence theorem for Gauss curvature flow
in high dimensions. J. Eur. Math. Soc. 19(12), 3735–3761 (2017)

\bibitem{Huang-Liu-Xu} Y. Huang, J. Liu, L. Xu, On the uniqueness of Lp-Minkowski problems: the constant p-curvature
case in R3. Adv. Math., 281: 906-927 (2015)

\bibitem{Huisken1990} G. Huisken, Asymptotic behavior for singularities of the mean curvature flow.
J. Differ. Geom. 31, 285–299 (1990)

\bibitem{McCoy2011} J. A. McCoy, Self-similar solutions of fully nonlinear curvature flows. Ann. Sc.
Norm. Super Pisa Cl. Sci. (5) 10, 317–333 (2011)

\bibitem{Ur-1} Urbas, J.: An expansion of convex hypersurfaces. J. Diff. Geom. 33 (1991), 91-125.
\end{thebibliography}
\end{document}